\documentclass{amsart}

\usepackage{a4,amssymb,amsfonts,amsmath,enumerate,color,url,hyperref}
\usepackage{enumerate}
\usepackage{amsthm}	
\usepackage{amsmath}
\usepackage{mathbbol}
\usepackage{eufrak}
\usepackage[utf8]{inputenc}
\numberwithin{equation}{section}
\newtheorem{theo}{Theorem}

\newtheorem{lemma}[theo]{Lemma}
\newtheorem{prop}[theo]{Proposition}

\newtheorem{cor}[theo]{Corollary}

\theoremstyle{definition}
\numberwithin{theo}{section}

\newtheorem{rem}[theo]{Remark}

\DeclareMathOperator{\Ker}{Ker}

\title[Diffusion on an interval under linear moment conditions]{Diffusion processes on an interval\\ under linear moment conditions}
\author{Delio Mugnolo}
\address{Delio Mugnolo, Institut f\"ur Analysis, Universit\"at Ulm, Helmholtzstra{\ss}e 18, 89081 Ulm, Germany}
\email{delio.mugnolo@uni-ulm.de}
\author{Serge Nicaise}

\address{Serge Nicaise, Universit\'e de Valenciennes et du Hainaut Cambr\'esis, LAMAV,  FR CNRS 2956, ISTV, F-59313 - Valenciennes Cedex 9, France}
\email{Serge.Nicaise@univ-valenciennes.fr}
\thanks{Version of \today\\
This article was written partially during a visit  of the second author at the LAMAV in Valenciennes under the financial support of the Land Baden--W\"urttemberg in the framework of the \emph{Juniorprofessorenprogramm} -- research project on ``Symmetry methods in quantum graphs''.}

\begin{document}

\begin{abstract}
We discuss a class of  {diffusion-type} partial differential equations on a bounded interval and discuss the possibility of replacing the boundary conditions by  {certain linear} conditions on the moments of order 0 (the total mass) and of another arbitrarily chosen order $n$. Each choice of $n$  induces the addition of a certain potential in the equation, the case of zero potential arising exactly in the special case of $n=1$  {corresponding to a condition on the barycenter}. In the linear case we exploit smoothing properties and perturbation theory of analytic semigroups to obtain well-posedness for the classical  heat equation (with said conditions on the moments). Long time behavior is studied for both the linear heat equation with potential and  {certain nonlinear equations of porous medium or fast diffusion type. In particular, we prove polynomial decay in the porous medium range and exponential decay in the fast diffusion range, respectively.}
\end{abstract}

\maketitle

\section{Introduction}

In~\cite{Vaz83} J.L.~V\'azquez made the simple observation that possibly diffusion-type equations of the form
\begin{equation}\label{pme1}
\frac{\partial u}{\partial t}(t,x)=\Delta(|u|^{p-2}u)(t,x),
\end{equation}
which are well-known to be associated with a well-posed Cauchy problem in $H^{-1}(\mathbb R^d)$ if $d=1$ (and even for $d>1$), enjoy conservation of mass and barycenter.  {Here and in the following, $p$ is some constant strictly larger than $1$: This result applies therefore to both the porous medium equation (PME) and the fast diffusion equation (FDE) along with the linear heat equation, corresponding to the cases of $p\in (2,\infty)$, $p\in (1,2)$ and $p=2$, respectively.}

V\'azquez' assertion is easily explained: For a density distribution function $f:\mathbb R\to \mathbb R$ denote by $\mu_n(f)$ the \textit{$n$-th moment}  (say, about 1), i. e., 
\[
\mu_n(f):=\int_\mathbb R (1-x)^n f(x)\; dx\qquad \hbox{for }f\in L^1(\mathbb R).
\]
Then in particular $\mu_0(f)$ and $\mu_1(f)$ represent the total mass and the barycenter of the distribution described by $f$, respectively. Now, differentiating with respect to time  {the moments of order 0 and 1 of a solution $u$ of~\eqref{pme1} with initial data $u_0$} and integrating by parts with respect to space a simple localization argument shows that
\[
\mu_0(u(t))=\mu_0(u_0)\qquad \hbox{and}\qquad \mu_1(u(t))=\mu_1(u_0)\quad\forall t\ge 0.
\]

Choosing boundary conditions judiciously, one can see that conservation of mass and/or barycenter may also hold in the case of a PME or a FDE on a bounded interval.  {Observe that conservation of mass and barycenter can also be defined for solutions so irregular that boundary conditions do not make sense -- this is particularly relevant in the case of the PME and the FDE, which are typically solved in spaces of distributions (\cite[Chapt.~10]{Vaz07}). May then the condition that mass and barycenter be conserved replace boundary conditions altogether?}

 {Imposing conditions on the moments may appear bizarre. However, since certain moment conditions boil down to boundary conditions if solutions are regular enough (\cite[Cor.~4.10]{MugNic11}), at a second glance it looks reasonable to investigate well-posedness under such conditions in the case of a bounded domain. Beginning with~\cite{Can63}, many authors have studied linear partial differential equations equipped with conditions on the moments complementing those on the boundary values.}
In~\cite{MugNic11} the present authors have gone on to observe that, in fact, in the case of the linear heat equation one can drop the boundary conditions and obtain well-posedness under a wide class of linear conditions on the moments of order 0 and 1 -- and in particular, whenever both of them are assumed to vanish constantly, i.e.,
\begin{equation}\label{lincond1}
\mu_n(u(t))=0\quad \forall t\ge 0,\; n=0,1.
\end{equation}
This special case had already been discussed by A.\ Bouziani and his coauthors starting with~\cite{BouBen96}, see~\cite{MugNic11} for more detailed references.  {A somehow comparable approach has been followed in~\cite[\S~9.6]{Vaz07}, where an analysis based on mere \emph{finiteness} of mass is performed.}

There exist counterexamples showing that, in general, moments of higher order are not conserved under the  evolution of~\eqref{pme1} on $\mathbb R$, cf.~\cite[\S~9.6.4]{Vaz07}. Hence it is not natural to expect well-posedness upon imposing a condition analogous to~\eqref{lincond1} for \emph{any} two moments. Our aim in this article is to show that, however, suitable  {conditions on $\mu_0$ and a further moment $\mu_n$} suffice to obtain well-posedness of certain modified evolution equations -- which can be looked at as PMEs or FDEs with a potential  {depending on $n$}.

It turns out that the analysis of  {diffusion equations} on an interval under conditions on $\mu_0$ and $\mu_n$ for general $n$ can be performed closely following some techniques developed in~\cite{MugNic11}. The extension of such techniques to the more general setting of the present article is discussed thoroughly in Section~\ref{sec:bouz}. As it is, the well-posedness results presented in~\cite{MugNic11} are just a special case of those that we obtain in Section~\ref{linearcase}.

However, serious problems seem to arise in the truly nonlinear case, i.e., whenever we discuss~\eqref{pme1} for $p\neq 2$ -- this is the topic of Section~\ref{nonlinearcase}.  {Both the PME and the FDE with Dirichlet boundary conditions are well-known to be the flow of the gradient associated with a suitable energy functional (also known as the functional's \emph{subdifferential} in the language of nonlinear semigroup theory, see e.g.~\cite[Example IV.6.B]{Sho97} and~\cite[Chapt.~10]{Vaz07}, and more generally~\cite{Bre73} for the abstract theory) with respect to an $H^{-1}$-inner product.}
(Observe that a different, more involved but also mightier approach based on flows on Riemannian manifolds has been introduced by F.\ Otto in a celebrated article~\cite{Ott01}). In our setting the gradient flow structure is still present, but unlike in the linear case we are not able to show that for initial data smooth enough this evolution equation is just the PME or the FDE with a certain potential. Nevertheless, we obtain well-posedness of a certain  {$n$-dependent} nonlinear evolution equation that strongly resembles the PME or the FDE, and we are able to show that its long-time behavior depends on $p$.  {In particular, we show that for all $n$ the decay of the $H^{-1}$-norm of the solutions is polynomial if $p\in (2,\infty)$ and exponential if $p\in (1,2]$. Our analysis is made different from the classical case by the fact that for this new evolution equation integration by parts is not easily applied -- this in turn prevents us from applying the classical method based on the weak formulation of the PME or the FDE, which are the backbone of many proofs in~\cite{Vaz07}.}

\section{The functional analytical setting}\label{sec:bouz}

If we consider $(0,1)$ as the torus $T$, then the test function set ${\mathcal D}(T)$ is in fact the set of smooth functions in $[0,1]$ such that the derivatives at all orders coincide at 0 and 1. In the same manner we will use the Sobolev space $H^1(T)$, by which we denote the subspace of those $u\in H^{1}(0,1)$ such that $u(0)=u(1)$ (i.e., of those $H^{1}$-functions supported on the torus). We use $L^2(0,1)$ as pivot space and denote by $H^{-1}(T)$ (resp., ${\mathcal D}'(T)$) the dual of $H^1(T)$ (resp., $\mathcal D(T)$).
 For $p\in (1,\infty)$, $p\not=2$, we define the  Sobolev spaces $W^{1,p}(T)$ and $W^{-1,p}(T)$ likewise. {In this paper we refrain from considering the case of~\eqref{pme1} for $p\in (0,1]$, which is known to require a quite different approach, cf.~\cite[\S~2.2.1]{Vaz06}.}

\begin{rem}\label{idmdef}
It was already observed in~\cite[Lemma~2.1]{MugNic11} that each element of ${\mathcal D}'(T)$ can be identified with an element of ${\mathcal D}'(0,1)$, but the identified vector is \emph{not} unique. We denote this non-injective identification operator by ${\rm Id}$, and by ${\rm Id}_m$ its restriction to $\{u\in H^{-1}(T):\mu_0(u)=0\}$ which, by~\cite[Lemma~2.4]{MugNic11}, is an isomorphism.
\end{rem}

For all $n\in \mathbb N_0$ we denote by $\mu_n$ the linear functional defined by
$$\mu_n(f):=\int_0^1(1-x)^n f(x)\,dx,$$
which is bounded on $L^1(0,1)$ (and even on $H^{-1}(T)$ for $n=0$).

For $n\in \mathbb N_0$ and $\varphi\in {\mathcal D}(T)$ we denote by $J_n \varphi$ the function defined by
\[
J_n \varphi(x):=\int_x^1 \varphi(y)\; dy-\mu_0(\varphi)(1-x)^n,\qquad x\in (0,1),
\]
and set for all $u\in H^{-1}(0,1)$
\begin{equation}\label{0}
\langle P_n u, \varphi\rangle:=\langle u, J_n \varphi\rangle\quad\forall \varphi\in {\mathcal D}(T).
\end{equation}
which is meaningful since $ J_n \varphi\in H^1_0(0,1)$.

We obtain the following analogue of~\cite[Lemma.~2.2]{MugNic11}, where the attention was devoted to $P_1$ only.

\begin{lemma}\label{lemma1}
Let $p\in [1,\infty)$ and $n\in \mathbb N$. Then the following assertions hold.
\begin{enumerate}
\item For all $u\in L^p(0,1)$, $P_nu$ can be written as
\begin{equation}\label{2}
P_nu(x)={\mathcal I} u(x)-\mu_n (u),\qquad \forall x\in (0,1),
\end{equation}
where
$$
{\mathcal I} u(x):=\int_0^x u(y)dy,\qquad \forall x\in (0,1).
$$
In particular, $(P_n u)'=u$ whenever $u\in L^p(0,1)$.
\item  Moreover, $P_n$ is a bounded linear operator from  $W^{-1,p}(T)$ to $L^p(0,1)$ and from $L^p(0,1)$ to $W^{1,p}(0,1)$ {as well as from $\{ f\in L^p(0,1):\mu_0(f)=0\}$ to $W^{1,p}(T)$}.
\item Finally, $(P_n u)'=u$  in $H^{-1}(0,1)$  for all $u\in H^{-1}(0,1)$.
\end{enumerate}
\end{lemma}

\begin{proof}
Let $u\in L^p(0,1)$ be fixed. Denote for a moment $\tilde P_nu$ the right-hand side of \eqref{2} (which is by construction an element of $W^{1,p}(0,1)$), i.e.,
$$
\tilde P_nu(x):={\mathcal I} u(x)-\mu_n (u),\qquad  x\in (0,1).
$$
By integrations by parts for all $\varphi\in {\mathcal D}(T)$ one has 
\begin{eqnarray*}
\int_0^1 \tilde P_nu(x) \varphi(x)\; dx &=& \int_0^1 ({\mathcal I}u(x)-\mu_n (u))\varphi(x)\; dx\\
&=&\int_0^1 \left(\int_0^x u(y)\; dy\right) \varphi(x)\; dx - \mu_n (u) \mu_0(\varphi).
\end{eqnarray*}
Accordingly, for $0\le y\le x\le 1$ by Fubini's Theorem we find
\begin{eqnarray*}
\int_0^1 \tilde P_nu(x) \varphi(x)\; dx
&=&\int_0^1 \left(\int_y^1 \varphi(x) \; dx\right) u(y)\; dy - \mu_n (u) \mu_0(\varphi)\\
&=&\int_0^1 \left(\int_y^1 \varphi(x) \; dx - \mu_0 (\varphi) (1-y)^n \right) u(y)\; dy 
\\
&=&\int_0^1 u(x)J_n\varphi(x)\; dx.
\end{eqnarray*}
This proves that $P_nu=\tilde P_nu$, hence (1),  and furthermore $P_n u\in  W^{1,p}(0,1)$.

In a second step we first easily check that for $\varphi\in {\mathcal D}(T)$,
$J_n\varphi$ is in $W^{1,q}(T)$ with $\frac1p+\frac1q=1$ (it is even in $W^{1,q}_0(0,1)$) and that
$$
\|J_n\varphi\|_{W^{1,q}_0(0,1)}\lesssim \|\varphi\|_{L^q(0,1)}.
$$
According to~\eqref{0} we then get
$$
| \langle P_nu, \varphi\rangle|\leq \| u\|_{W^{-1,p}(T)} \|J_n\varphi\|_{W^{1,q}_0(0,1)}\lesssim \| u\|_{W^{-1,p}(T)} \|\varphi\|_{L^q(0,1)}.
$$

Therefore
$$\|P_n u\|_{L^p(0,1)}\lesssim \|u\|_{W^{-1,p}(T)},\qquad \forall u\in {\mathcal D}(T),$$
which proves that $P_n$ is a bounded operator from $W^{-1,p}(T)$ to $L^p(0,1)$
due to the density of ${\mathcal D}(T)$ into $W^{-1,p}(T)$.
As the boundedness of $P_n$ from $L^p(0,1)$ to $W^{1,p}(0,1)$ has been observed above, point (2) is proved.

Finally, the assertion (3)  follows the definition of $P_n$ and the fact that $J_n(\varphi')=-\varphi$
for all $\varphi\in H^1_0(0,1)$.
\end{proof}

In particular, 
\[
(P_n 1)(x)=x-\int_0^1 (1-x)^n \; dx=x-\frac{1}{n+1}.
\]

Observe that, unlike for $n=1$, for $n\ge 2$ one has in general 
\begin{equation}\label{nogut}
P_n\not\in {\mathcal L}\Big(\{ f\in H^{-1}(T):\mu_0(f)=0\},\{ f\in L^p(0,1):\mu_0(f)=0\}\Big),
\end{equation}
since $\langle P_n u,1\rangle\not=0$.

{
\begin{rem}\label{rem-mugnic11}
Recall that by~\cite[Lemma~2.3]{MugNic11} 
\[
(P_1 u|P_1 v)+\mu_0(u)\overline{\mu_0(v)},\quad u,v\in H^{-1}(T),
\]
defines an equivalent inner product on $H^{-1}(T)$, and in particular
\begin{equation}
\label{equivnorm}
\| {\rm Id}_m^{-1}u\|_{H^{-1}(T)}\simeq \|P_1 {\rm Id}_m^{-1}u\|_{L^2(0,1)},\quad \forall u\in H^{-1}(T).
\end{equation}
In fact, we can say more. Taking into account Remark~\ref{munmun1}.(1) and reasoning just like in the proof of~\cite[Lemma~2.3]{MugNic11}  one can easily see that 
\begin{equation*}\label{1}
 \|u\|_{H^{-1}(T)} \lesssim \|P_n u\|_{L^2}+|\mu_0(u)|, \quad \forall u\in H^{-1}(T).
\end{equation*}
Accordingly, we can even say that 
 \[
(P_n u|P_n v)+\mu_0(u)\overline{\mu_0(v)},\quad  u,v\in H^{-1}(T),
\]
defines an equivalent inner product on $H^{-1}(T)$ for all $n\in \mathbb N$.
\end{rem}}

{
\begin{rem}\label{munmun1}
Let $n\in \mathbb N$. 
\begin{enumerate}
\item Observe that for all $n\in \mathbb N$
$$\mu_n(f)=n\int_0^1 (1-x)^{n-1} \int_0^x f(z)dz dx=n\mu_{n-1}({\mathcal I}f)\quad \forall f\in L^1(0,1).$$
\item {Let $f\in L^1(0,1)$. Then by Lemma~\ref{lemma1}.(1) and integrating by parts
\begin{eqnarray*}
\mu_{n-1}(P_n f)&=&\int_0^1 (1-x)^{n-1} {\mathcal I}f(x)\; dx-\mu_n(f)\int_0^1 (1-x)^{n-1}dx\\
&=&\left[ \frac{(1-x)^n}{n} \int_0^x f(x)dx\right]_{x=0}^{x=1}-\int_0^1 \frac{(1-x)^n}{n}f(x)\; dx+\frac{1}{n}\mu_n(f)\\
&=&-\frac{1}{n}\mu_n(f)+\frac{1}{n}\mu_n(f)=0.
\end{eqnarray*}}
\item Observe also that
\begin{equation*}\label{sergeajout3}
 P_n \delta_1=0 \qquad \hbox{in }{\mathcal D}'(T),
\end{equation*}
since for all $\phi \in {\mathcal D}(T)$ $\langle P_n\delta_1, \phi\rangle=\langle \delta_1,J_n\phi\rangle = J_n\phi(1)=0$.
\item It is also important for the following that
\begin{equation}\label{Pn02}
P_n h(0)=-\mu_n(h)\qquad \hbox{and}\qquad P_n h(1)=\mu_0(h)-\mu_n(h)\quad \forall h\in L^1(0,1)
\end{equation}
as well as
\begin{equation}\label{Pn01}
\mu_0(P_n(h))=\int_0^1 P_n h(x)\; dx=\mu_1(h)-\mu_n(h)\quad \forall h\in L^1(0,1):
\end{equation}
both identities are direct consequences of~\eqref{2}.
\item It is a straightforward observation that~\cite[Rem.~2.6]{MugNic11} can be generalized as follows: For all $p\in [1,\infty)$, all $f\in W^{2,p}(0,1)$ and all $\psi\in L^p(0,1)$, we have
\begin{equation}\label{ajoutserge10}
{\rm Id}_m^{-1}(f''+\psi)=f''+\psi +(\mu_0(f''+\psi))\delta_1 \quad\hbox{ in } W^{-1,p}(T),
\end{equation}
and in particular by (3)
$$
P_n({\rm Id}_m^{-1}(f''+\psi))=P_n(f''+\psi) \quad\hbox{ in } L^p(0,1).
$$
\item Finally, observe that by Lemma~\ref{lemma1} and~\eqref{Pn02} $P_n$ is bounded {from $\{f\in L^p(0,1):\mu_0(f)=0\}$ to $W^{1,p}(T)$} and from $\{f\in L^p(0,1):\mu_0(f)=\mu_n(f)=0\}$ to $W^{1,p}_0(0,1)$.
\end{enumerate}
\end{rem}

The crucial point for our investigation is that an integration-by-parts-type formula holds. Recall that we are denoting by ${\rm Id}_m$ the isomorphism between $\{u\in H^{-1}(T):\mu_0(u)=0\}$ and $H^{-1}(0,1)$ and let $u\in H^1(0,1)$. Then for all $u\in H^1(0,1)$ and all $h\in H^{1}(T)$, by Remark~\ref{rem-mugnic11} one has
\begin{eqnarray*}
\left({\rm Id}_m^{-1}(u'')|h\right)_{H^{-1}(T)}&=& (P_n {\rm Id}_m^{-1}(u'')|P_n h)_{L^2}\\
&=&(u'-a|P_n h)_{L^2},
\end{eqnarray*}
where
\begin{equation}\label{adef}
a:=\langle u'-P_n (u''),1\rangle.
\end{equation}
{Hence, \eqref{Pn01} and a standard integration by parts yield
\begin{eqnarray}
\nonumber
\left({\rm Id}_m^{-1}(u'')|h\right)_{H^{-1}(T)}&=&(u'|P_n h)_{L^2}-a\overline{\left(\mu_1(h)-\mu_n(h)\right)}\\
&=&-(u| h)_{L^2}+[u P_n \overline{h}]_0^1 -a\overline{\left(\mu_1(h)-\mu_n(h)\right)}\label{serge30/08}
\end{eqnarray}
since by Lemma~\ref{lemma1}.(3) $(P_n h)'=h$.}

\begin{rem}
Note that
\[
a= u(1)-u(0)-\langle u'', (1-{\rm id})-(1-{\rm id})^n\rangle
\]
because
\[
\langle P_n (u''),1\rangle = \langle u'',J_n 1\rangle
\]
-- this is the duality between $H^{-1}(0,1)$ and $H^1_0(0,1)$, with the function $J_n 1\in H^1_0(0,1)$ given by
\[
(J_n 1)(x)=(1-x) - (1-x)^n\quad \forall x\in (0,1).
\]
For $n=1$, $J_n 1\equiv 0$ and  we get $\langle P_n (u''),1\rangle =0$,
but in any case, this is does not really matter since $a$ is multiplied by $0$ in \eqref{serge30/08}.
\end{rem}

 Indeed, we can still improve the formula in~\eqref{serge30/08} for $n\ge 2$.

\begin{theo}\label{ibp}
Let $u\in H^1(0,1)$ and $h\in L^2(0,1)$. Then
\[
\left({\rm Id}_m^{-1}(u'')|h\right)_{H^{-1}(T)}= -(u|h)_{L^2}+\left(\begin{pmatrix}
u(1)\\ n u(0)- n (n-1)\mu_{n-2}(u)\\
(1-n) u(0)-u(1)+n (n-1)\mu_{n-2}(u)\\  
\end{pmatrix}\Big|\begin{pmatrix}
\mu_0(h)\\
 \mu_1(h)\\ \mu_n (h)
\end{pmatrix}\right)_{\mathbb C^3}.
\]
\end{theo}

\begin{proof}
The case $n=1$ has already been discussed. For $n\geq 2$, our starting point is again~\eqref{serge30/08}. From the  property $J_n 1\in H^1_0(0,1)$ and an integration by parts, we can write
\begin{eqnarray*}
\langle P_n (u''),1\rangle &=& -\langle u', (J_n 1)'\rangle\\
&=& \int_0^1 u'(x) (1-n(1-x)^{n-1})\,dx\\
&=& u(1)-u(0)-n (n-1)\int_0^1 u(x)   (1-x)^{n-2})\,dx-n[u(x) (1-x)^{n-1}]_0^1
\\
&=& u(1)-u(0)-n (n-1)\mu_{n-2}(u)+n u(0).
\end{eqnarray*}
Hence we deduce that for $n\geq 2$, the constant in~\eqref{adef} is given by
\[
a=n (n-1)\mu_{n-2}(u)-n u(0).
\]

Now, applying~\eqref{Pn02} we find
\begin{eqnarray*}
[u  \overline{P_n h}]_0^1 &=& u(1)\overline{P_n h(1)}-u(0) \overline{P_n h(0)}\\
&=& u(1)\overline{\left( \mu_0(h)-\mu_n(h)\right)}+u(0) \overline{\mu_n (h)}\\
&=& u(1) \overline{\mu_0(h)}+\left(u(0)-u(1)\right)\overline{\mu_n(h)}.
\end{eqnarray*}
Applying the previous identities to~\eqref{serge30/08} we obtain
\begin{eqnarray*}
&&\left({\rm Id}_m^{-1}(u'')|h\right)_{H^{-1}(T)}= -(u|h)_{L^2}+u(1) \overline{\mu_0(h)}+(u(0)-u(1))\overline{\mu_n(h)}\\
&&\qquad \qquad \qquad \qquad \qquad \qquad -\big(n (n-1)\mu_{n-2}(u)-n u(0)\big)
\overline{\left(\mu_1(h)-\mu_n(h)\right)},
\end{eqnarray*}
and the claimed formula follows.
\end{proof}

If $n=1$, this is the formula already obtained in~\cite[Lemma~2.13]{MugNic11}. Otherwise, we have obtained a more general identity that allows us to extend the study of diffusion processes under linear conditions on $\mu_0$ and $\mu_1$ (as e.g.\ in~\cite{BouBen96,MugNic11}) to linear conditions on $\mu_0$ and $\mu_n$ for general $n\in \mathbb N$. 

For each subspace $Y$ of $\mathbb C^2$ and each $p>1$ we can consider the reflexive Banach space
$$V^{(1)}_{Y,p}:=\left\{f\in L^p(0,1):\begin{pmatrix}
\mu_0(f)\\ \mu_1(f)
\end{pmatrix}\in Y\right\}.$$

\begin{lemma}\label{lemma:surj}
Let $k\in \mathbb N$. Then for all $m\in \mathbb N^k$ with pairwise distinct entries, the operator  
\[
\begin{pmatrix}
\mu_{m_1}  \\ \vdots\\ \mu_{m_k}
\end{pmatrix}:C^\infty[0,1]\to \mathbb C^k
\]
 is surjective.
\end{lemma}

\begin{proof}
We are going to prove more, namely that the vector valued mapping
\[
\begin{pmatrix}
\mu_{0}  \\ \vdots\\ \mu_{m_k}
\end{pmatrix}
\]
is surjective from the space $\mathbb P_{m_k}(0,1)$ of all polynomials of degree at most $m_k$ to $\mathbb C^{m+1}$.
For $m\in \mathbb N$ fixed, it is well known that there exists a  $(m+1)\times (m+1)$ invertible matrix $M_m$ such that
\[
\begin{pmatrix}
1\\ (1-id)  \\ \vdots\\ (1-id)^m
\end{pmatrix}=M_m\begin{pmatrix}
Q_0\\ Q_1  \\ \vdots\\ Q_m
\end{pmatrix},
\]
where $Q_k$ is  the Legendre-type polynomial of degree $k$ defined on $(0,1)$
through the standard Legendre  polynomial $P_k$ of degree $k$ defined on $(-1,1)$ by
$$
Q_k(x)=P_k(2x-1), \quad\forall x\in (0,1).
$$
Setting
$$
\tilde\mu_{k}(h):=\int_0^1 h(x)Q_k(x)\,dx,\quad h\in L^1(0,1),
$$
we deduce that
\begin{equation}\label{serge26/11}
\begin{pmatrix}
\mu_{0}  \\ \vdots\\ \mu_{m}
\end{pmatrix}=M_m\begin{pmatrix}
\tilde\mu_{0}  \\ \vdots\\ \tilde\mu_{m}
\end{pmatrix}.
\end{equation}

As the Legendre  polynomials are pairwise orthogonal, the mapping
\[
\begin{pmatrix}
\tilde\mu_{0}  \\ \vdots\\ \tilde\mu_{m}
\end{pmatrix}:{\mathbb P_m}(0,1)\to \mathbb C^{m+1}
\]
is trivially surjective. By \eqref{serge26/11} and the invertibility of $M_m$,
we conclude that
\[
\begin{pmatrix}
\mu_{0}  \\ \vdots\\ \mu_{m}
\end{pmatrix}:{\mathbb P_m}(0,1)\to \mathbb C^{m+1}
\]
is also surjective.
\end{proof}


We denote by $H_Y$ the Hilbert space 
\begin{equation}
\label{defi_hy}
H_Y:=\left\{
\begin{array}{ll}
\{f\in H^{-1}(T):\mu_0(f)=0\},\qquad &\hbox{if }Y=\{0\}^2\hbox{ or }Y=\{0\}\times {\mathbb C},\\
H^{-1}(T), & \hbox{otherwise}.
\end{array}
\right.
\end{equation}
It has been shown in~\cite{MugNic11} that an equivalent inner product is given in either case by
\begin{equation}
\label{prodH}
(f|g)_{H_Y}:=\int_0^1 P_1 f(x) \overline{P_1 g(x)} dx+\mu_0(f)\overline{\mu_0( g)}.
\end{equation}

\begin{lemma}\label{lemmadensity1}
The space $V^{(1)}_{Y,p}$ is dense in $H_Y$ for each $Y$ subspace of $\mathbb C^2$.
\end{lemma}

\begin{proof}
By construction, ${V^{(1)}_{Y,p}}\subset H_Y$ and hence the inclusion $\overline{V^{(1)}_{Y,p}}\subset H_Y$ is clear. 

The proof of the converse inclusion is divided in several steps, but follows closely~\cite[Cor.~2.4 and Thm.~4.2]{MugNic11}. We first prove the assertion for $Y=\{0\}^2$. 

\noindent
(1) To begin with, observe that it follows from~\eqref{prodH} that
\begin{equation}
\label{prodHbis}
(f|g)_{H_{\{0\}^2}}=\int_0^1 P_1 f(x) \overline{P_1 g(x)} dx,\qquad \forall f,g\in H_{\{0\}^2}.
\end{equation}
and 
\begin{equation}\label{4}
\|u\|_{H^{-1}(T)}\simeq \|P_1 u\|_{L^2}\quad\forall u\in H_{\{0\}^2}.
\end{equation}
To prove the claimed inclusion, we show that each $f\in H_{\{0\}^2}$ that is orthogonal to $\overline{V^{(1)}_{\{0\}^2,p}}$
for the inner product $(\cdot|\cdot)_{H_{\{0\}^2}}$ is identically zero. 
In fact, let $f\in H_{\{0\}^2}$ be orthogonal to $\overline{V^{(1)}_{\{0\}^2,p}}$. Then it satisfies
$$
(P_1 f| P_1 g)_{L^2}=0,\qquad \forall g\in  V^{(1)}_{\{0\}^2,p}.
$$
But according to its definition (\ref{0}) we get equivalently
\begin{equation}\label{5}
\langle f, J_1 P_1g\rangle=0,\qquad \forall g\in  V^{(1)}_{\{0\}^2,p}.
\end{equation}
Now define
\begin{equation}
\label{definW}
W:=\{u\in H^1(T)\cap W^{2,p}(0,1): u'(0)=u'(1)=0\}.
\end{equation}
This space is dense in $H^1(T)$ since for $u\in H^1(T)$, 
$u-u(0)$ belongs to $H^1_0(0,1)$.
Hence for all $u\in H^1(T)$ there exists a sequence of $\varphi_n\in {\mathcal D}(0,1)$ such that
$$
\varphi_n\to u-u(0) \quad \hbox{ in } H^1_0(0,1),
$$
and therefore
 $\varphi_n+u(0)\in W$ with
$$
\varphi_n+u(0)\to u \quad \hbox{ in } H^1(T).
$$

Now for $h\in W$, we take $g:=-h''\in V^{(1)}_{\{0\}^2,p}$. By construction
$$
J_1 P_1g(x)=h(x)-h(1) ,\qquad \forall x\in (0,1).
$$
Plugging this identity in~\eqref{5} yields
\begin{equation}\label{sergeortho}
\langle f, h-h(1)\rangle=0,\qquad \forall h\in W.
\end{equation}
But due to the fact that $\mu_0( f)=0$, we deduce that
$$
\langle f, h\rangle=0,\qquad \forall h\in W.
$$
As $W$ is dense in $H^1(T)$, we conclude that $f=0$.

If $Y=\{0\}\times {\mathbb C}$, we see that
$$H_Y=\overline{V^{(1)}_{\{0\}^2,p}}\subset \overline{V^{(1)}_{\{0\}\times \mathbb C,p}}\subset H_Y,$$
and the assertion follows. 

(2) We consider the remaining cases. More precisely, it suffices to study the case of 1-dimensional $Y$ with $Y\not=\{0\}\times \mathbb C$, because once proved this for $Y=\mathbb C^2$ and any 1-dimensional subspace $Y_0$ we have
$$H_Y=\overline{V^{(1)}_{Y_0,p}}\subset \overline{V^{(1)}_{\mathbb C^2,p}}=\overline{L^p(0,1)}= H_Y.$$

Hence, assume that there exists $\alpha \in \mathbb R$ such that $Y$ is the set of all $(z_0,z_1)\in \mathbb C^2$ satisfying
\begin{equation}\label{carY}
 z_1=\alpha z_0.
\end{equation}
Let $f\in H^{-1}(T)$ be such that
\begin{equation}\label{orthoV1bis}
(f|g)_{H_Y}=0,\qquad \forall g\in V^{(1)}_{Y,p}.
\end{equation}
Since $V^{(1)}_{\{0\}^2,p}\subset V^{(1)}_{Y,p}$, one has
$$
(f|g)_{H^{-1}(T)}=0,\qquad \forall g\in V^{(1)}_{\{0\}^2,p},
$$
and reasoning as in (1) we deduce that
(\ref{sergeortho}) holds for the space $W$ defined in~\eqref{definW}. Since $W$ is dense in $H^1(T)$, this is equivalent to
$$
\langle f, h-h(1)\rangle=0,\qquad \forall h\in H^1(T),
$$
or again
\begin{equation}\label{sergestar}
f=\mu_0( f)\delta_1 \hbox{ in } H^{-1}(T).
\end{equation}
Coming back to (\ref{orthoV1bis}) and taking into account Remark~\ref{rem-mugnic11} we get
\begin{equation}\label{mu0}
0= \mu_0( f) \left((P_1 \delta_1|P_1 g)_{L^2}+\mu_0( \delta_1)\overline{\mu_0(g)}\right)=\mu_0( f) \overline{\mu_0(g)},\qquad \forall g\in V^{(1)}_{Y,p}.
\end{equation}
Now, by Lemma~\ref{lemma:surj} there exists $g\in L^p(0,1)$ such that
$$
\begin{pmatrix}
\mu_0(g) \\ \mu_1(g)
\end{pmatrix}
=
\begin{pmatrix}
1\\ \alpha
\end{pmatrix},
$$
where $\alpha$ is as in~\eqref{carY}. Hence, we can plug such a $g\in V^{(1)}_{Y,p}$ in~\eqref{mu0} and we find $\mu_0( f)=0$, hence, by~\eqref{sergestar}, $f=0$. This concludes the proof.
\end{proof}

\begin{lemma}\label{lemmadensity2}
Let $n\in \mathbb N$. Then the vector space
\begin{equation}
\label{defi:V^n}
V^{(n)}_{\{0\}^2,p}:=\{u\in L^p(0,1):\mu_0(u)=\mu_n(u)=0\},
\end{equation}
and hence also
$$V^{(n)}_{\{0\}\times \mathbb C,p}:=\{u\in L^p(0,1):\mu_0(u)=0\},$$
are dense in 
\[
H:=\{f\in H^{-1}(T):\mu_0(f)=0\}.
\]
\end{lemma}

\begin{proof}
The proof is based on Lemma~\ref{lemmadensity1}, i.e., on the validity of the same assertion in the special case of $n=1$. Let $w\in H$: we are looking for a sequence $(w_k)_{k\in \mathbb N}\subset {\Ker}\mu_n$  here and below in the proof, $\mu_n$
is seen as a mapping from $H\cap L^p(0,1)$ into $\mathbb C$) that approximates $w$ in $H^{-1}(T)$. By Lemma~\ref{lemmadensity1}, we already know that this is possible for $n=1$: that is, there exists a sequence $(v_k)_{k\in \mathbb N}\subset {\Ker}\mu_1$ that approximates $w$ in $H^{-1}(T)$. Now,  for $n\geq 2$ observe that the binomial formula yields
\[
(1-t)^n=(1-t)+\tilde{p}(t),\qquad t\in (0,1),
\]
with $\tilde{p}\in H^1(T)$, and accordingly 
\[
\mu_n(v)=\mu_1(v)+\langle \tilde{p},v\rangle_{H^1(T),H^{-1}(T)}\quad \forall v\in L^1(0,1),
\]
and in particular
\[
0=\mu_1(v_k)=\mu_n(v_k)-\langle \tilde{p},v_k\rangle_{H^1(T),H^{-1}(T)}\quad \forall k\in \mathbb N.
\]
Thus, by continuity,
\[
\lim_{k\to \infty} \mu_n (v_k)=\lim_{k\to \infty}\langle \tilde{p},v_k\rangle =\langle \tilde{p},w\rangle=:c.
\]
We distinguish the two cases $c=0$ and $c\neq 0$.
\begin{enumerate}
\item If $c=0$, it suffices to take 
\[
w_k:=v_k -\frac{\mu_n(v_k)}{\mu_n (\tilde{p})}\tilde{p}
\]
(observe that $\mu_n(\tilde{p})=\int_0^1 |(1-t)|^{2n}dt\neq 0$),
so that $w_k \in \Ker \mu_n$ for all $k\in \mathbb N$ and moreover
\[
\lim_{k\to \infty} w_k=w.
\]
\item If $c\neq 0$, write $w=w_1+\alpha v_0$, with  $v_0 \in \Ker \mu_n$ such that $\mu_1(v_0)=\langle \tilde{p},v_0\rangle\neq 0$, i.e., $v_0 \in \Ker \mu_n \setminus \Ker \mu_1$. Now, set 
\[
\alpha:=\frac{\langle \tilde{p}, w\rangle}{\langle \tilde{p}, v_0\rangle }
\]
and observe that for $w_1:=w-\alpha v_0$ one has $\langle \tilde{p}, w_1\rangle=0$, hence   $w_1\in \overline{\Ker \mu_n}$  owing to the first case.     As $v_0 \in \Ker \mu_n$
we conclude that $w\in \overline{\Ker \mu_n}$ as well. 
\end{enumerate}
This concludes the proof.
\end{proof}

Also the following holds.

\begin{lemma}\label{lemmadensity3}
Let $n\in \mathbb N$ and $Y$ be a subspace of $\mathbb C^2$ with $Y\not=\{0\}^2$ and $Y\not= \{0\}\times \mathbb C$. Then the vector space
$$V^{(n)}_{Y,p}:=\left\{u\in L^p(0,1):\begin{pmatrix}\mu_0(u)\\\mu_n(u)\end{pmatrix}\in Y\right\}$$
is dense in $H^{-1}(T)$.
\end{lemma}

\begin{proof}
The proof closely follows that of~\cite[Thm.~4.2]{MugNic11}. Like in that proof, we begin by observing that if $Y\not=\{0\}^2$ and $Y\not= \{0\}\times \mathbb C$, then either $Y=\mathbb C^2$ or there exists $\alpha \in \mathbb C$ such that $Y$ is the set of all $(z_0,z_1)\in \mathbb C^2$ satisfying
\begin{equation}\label{carYb}
 z_1=\alpha z_0.
\end{equation}

Let us first assume that $Y\not=\mathbb C^2$ and let $f\in H^{-1}(T)$ be such that
\begin{equation}\label{orthoV1}
(f|g)_{H^{-1}(T)}=0,\qquad \forall g\in V^{(n)}_{Y,p},
\end{equation}
and in particular
\begin{equation*}\label{orthoV1bis-c}
(f|g)_{H^{-1}(T)}=0,\qquad \forall g\in V^{(n)}_{\{0\}^2,p}.
\end{equation*}
Just like in the proof of Corollary~\ref{lemmadensity1} we deduce that~\eqref{sergeortho} holds for the space $W$ defined in~\eqref{definW}. Since $W$ is dense in $H^1(T)$, this is equivalent to
$$
\langle f, h-h(1)\rangle=0,\qquad \forall h\in H^1(T),
$$
or again
\begin{equation}\label{sergestar-b}
f=\mu_0( f)\delta_1 \quad\hbox{ in } H^{-1}(T).
\end{equation}
Plugging into~\eqref{orthoV1} we get
\begin{equation}\label{mu0-b}
0= \mu_0( f) \left((P_1\delta_1|P_1g)_{L^2}+\mu_0( \delta_1)\overline{\mu_0( g)}\right)=\mu_0( f) \overline{\mu_0(g)},\qquad \forall g\in V^{(n)}_{Y,p},
\end{equation}
because by Remark~\ref{munmun1}.(3) $P_1\delta_1=0$. By Lemma~\ref{lemma:surj} there exists $g\in L^p(0,1)$ such that
$$
\begin{pmatrix}
\mu_0(g) \\ \mu_n(g)
\end{pmatrix}
=
\begin{pmatrix}
1\\ \alpha
\end{pmatrix},
$$
where $\alpha$ is as in~\eqref{carYb}. Hence, we can plug such a $g\in V^{(n)}_{Y,p}$ in~\eqref{mu0-b} and we find $\mu_0( f)=0$, hence, by~\eqref{sergestar-b}, $f=0$. This shows that $V^{(n)}_{Y,p}$ is dense in $H^{-1}(T)$ if $Y\not=\mathbb C^2$. For the case $Y=\mathbb C^2$, however, by definition $V^{(n)}_{Y,p}=L^p(0,1)$, and the claimed density is clear.
\end{proof}

Summing up, we have proved the following $n$-th-moment-analogue of Lemma~\ref{lemmadensity1}.

\begin{cor}\label{lemmadensityn}
Let $n\in \mathbb N$ and $Y$ be a subspace of $\mathbb C^2$. The space $V^{(n)}_{Y,p}$ is dense in $H_Y$ for each $Y$ subspace of $\mathbb C^2$.
\end{cor}

\section{The linear case}\label{linearcase}

It is instructive to consider the linear case of $p=2$ first.

\begin{lemma}
Let $n\in \mathbb N$ and $Y$ be a subspace of $\mathbb C^2$. Then the sesquilinear form
 $$a(u,v):=\int_0^1 u(x)\overline{v(x)}\; dx,\qquad u,v\in V^{(n)}_{Y,2},$$
is densely defined, symmetric, continuous and coercive in $H_Y$.
\end{lemma}

Then, well-posedness of first and second order abstract Cauchy problem associated with this form follow directly, by means of the general theory of quadratic forms.

\begin{cor}
Let $n\in \mathbb N$ and $Y$ be a subspace of $\mathbb C^2$.  Then the operator $(A_Y ,D(A_Y))$ associated with $(a,V^{(n)}_{Y,2})$ is positive definite on $H$. In particular, $-A_Y $ generates a cosine operator function with associated phase space $V^{(n)}_{Y,2}\times H_Y$ and hence an exponentially stable, contractive, analytic semigroup $(e^{-tA_Y})_{t\ge 0}$ of angle $\frac{\pi}{2}$ on $H_Y$. This semigroup is immediately of trace class.
\end{cor}

\subsection{The case of $\mu_0(u)=0$}\label{sec:mu0}

It remains to determine the operator associated with the form, and hence the Cauchy problems actually solved by the cosine operator functions and the operator semigroup. To begin with, we consider the case where we impose $\mu_0(u)=0$, corresponding to setting $Y=\{0\}^2$ or $Y=\{0\}\times \mathbb C$.

\begin{theo}\label{identK}
{Let $n\in \mathbb N$ and define a bounded linear functional $\gamma:H^1(0,1)\to \mathbb C$ by
\begin{eqnarray}
\gamma(f):=
\left\{
\begin{array}
{ll}
(n-1)(2n-1) f(0)-(n-1)^2(2n-1)\mu_{n-2}(f),\quad &\hbox{if }n\ge 2,\\
0 &\hbox{if }n=1.
\end{array}
\right.
\label{serge30/08b}
\end{eqnarray}}
If either $Y=\{0\}^2$ or $Y=\{0\}\times \mathbb C$, then the operator $(A_Y ,D(A_Y))$ associated in $H_Y$ with the quadratic form $a$ has domain given by
\begin{eqnarray*}
D(A_Y)&=&\{u\in H^1(0,1):\mu_0(u)=\mu_n(u)=0\}\qquad \hbox{or}\\
D(A_Y)&=&\{u\in H^1(0,1):\mu_0(u)=0,\; u(0)=u(1)\}
\end{eqnarray*}
respectively, and its action is given in both cases by
\[
A_Y u={\rm Id}_m^{-1}(-u''+\gamma(u) (1-{\rm id})^{n-2}).
\]
\end{theo}

We recall that ${\rm Id}_m$ is the isomorphism introduced in Remark~\ref{idmdef}. The above theorem shows that in particular
$$
A_Y u=-u''+\gamma(u) (1- id)^{n-2} \quad\hbox{ in } {\mathcal D}'(0,1).
$$

\begin{proof}
{
We only consider the case of $n\ge 2$, as the case of $n=1$ has been proved in~\cite[Thm.\ 3.1 and 3.3]{MugNic11}.}
Denote
$${\mathcal K}:=\{u\in H^1(0,1):\mu_0(u)=\mu_n(u)=0\}.$$
Let us first show the inclusion $D(A_Y)\subset {\mathcal K}$. Let $f\in D(A_Y)$. Then there exists $g\in H^{-1}(T)$ for which $\mu_0( g)=0$ (i.e., $g\in H_Y$) and such that
$$
(f|h)_{L^2}=a(f,h)\stackrel{!}{=}(g|h)_{H_{\{0\}^2}}=\int_0^1 (P_n g)(x)\overline{(P_n h)(x)}\,dx\quad\forall h\in V^{(n)}_{Y,2},
$$
by virtue of Remark~\ref{rem-mugnic11}. Now for $g\in H^{-1}(T)$ it follows from Lemma~\ref{lemma1}.(1) that $P_n g\in L^2(0,1)$, hence $(P_n(P_ng))'=P_n g$.
Integrating by parts  we obtain that in particular for all $h\in V^{(n)}_{\{0\}^2,2}$ 
\begin{eqnarray*}
\int_0^1 (P_n g)(x)\overline{(P_n h)(x)}\,dx&=&\int_0^1 (P_n (P_n g))'(x)\overline{(P_n h)(x)}\,dx\\&=&
-\int_0^1 (P_n (P_n g))(x)\overline{{(P_n h)'}(x)}\,dx\\
&=&-\int_0^1 (P_n (P_n g))(x)\overline{{h}(x)}\,dx.
\end{eqnarray*}
 This shows that
\begin{equation*}
(f|h)_{L^2}=-(P_n (P_n g)|h)_{L^2}\quad\forall h\in V^{(n)}_{\{0\}^2,2}
\end{equation*}
because of~\eqref{Pn02}. 
Let us now denote by $\Pi$ the orthogonal projection of $L^2(0,1)$ onto the closed subspace of polynomials of one variable  spanned by $1$ and $(1-{\rm id})^n$. Then we have
$$
f=-(I-\Pi)P_n (P_n g)+ \Pi f=-P_n (P_n g)+\Pi (P_n(P_n g)+f).
$$
Accordingly, in either case $f\in H^1(0,1)$ in view of Lemma~\ref{lemma1}.(2). Moreover,
$$
f''=-g+\left(\Pi P_n (P_n g)\right)'' \hbox{ in } {\mathcal D}'(0,1).
$$
As 
$$\Pi \left(P_n (P_n g)+f\right)(x)=\alpha+\beta (1-x)^n,
$$
for some $\alpha, \beta\in\mathbb R$, we have found
that
$$
f''=-g+n(n-1) \beta x^{n-2} \hbox{ in } {\mathcal D}'(0,1).
$$
By~\cite[Lemma ~2.4]{MugNic11} this yields
that 
$$A_Y f={\rm Id}_m^{-1}(-f''+n(n-1) \beta (1-x)^{n-2}).
$$
Clearly $\beta$ depends on $f$ and its dependence will be given below.

Let us finally check that the additional conditions in the definition of $D(A_Y)$ hold. For all $h\in V^{(n)}_{Y,p}$
\begin{eqnarray*}
(f|h)_{L^2}&=&a(f,h)\\
&\stackrel{!}=&(A_Y f|h)_{H_Y}\\
&=&\left({\rm Id}_m^{-1}(-f''+\gamma(f) (1-id)^{n-2})|\ h\right)_{H^{-1}(T)}\\
&=&\left({\rm Id}_m^{-1}(k'')|  h\right)_{H^{-1}(T)}
\end{eqnarray*}
where 
\[
-k(x):=f(x)-\frac{1}{n(n-1)} \gamma(f) (1-x)^n)),\quad x\in (0,1),
\]
for some $\gamma(f)\in \mathbb C$ to be determined below.
Hence in view of Theorem~\ref{ibp}, and since $\mu_0(h)=0$, we find that
\begin{eqnarray*}
(f|h)_{L^2}&=&\left({\rm Id}_m^{-1}(k'')| h\right)_{H^{-1}(T)}\\
&=&-(k|h)_{L^2}+\left(\begin{pmatrix}
n k(0)- n (n-1)\mu_{n-2}(k)\\
(1-n) k(0)-k(1)+n (n-1)\mu_{n-2}(k)\\  
\end{pmatrix}\Big|\begin{pmatrix}
 \mu_1(h)\\ \mu_n (h)
\end{pmatrix}\right)_{\mathbb C^2}
\end{eqnarray*}
Observe that 
\[
-(k|h)_{L^2}=(f|h)_{L^2}-\frac{1}{n(n-1)}\gamma(f)\overline{\mu_{n}(h)}
\]
and furthermore due to Lemma~\ref{lemma:surj}
\begin{equation}\label{systeme0}
0=- k(0)+  (n-1)\mu_{n-2}(k)
\end{equation}
in case $Y=\{0\}^2$ (and hence $\mu_n(h)=0$); or else
\begin{equation}\label{systeme1}
\left\{
\begin{array}{rcl}
 nk(0)&=& n(n-1)\mu_{n-2}(k)\\
\frac{1}{n(n-1)}\gamma(f)&=&(1-n) k(0)-k(1)+n (n-1)\mu_{n-2}(k)
\end{array}
\right.
\end{equation}
if $Y=\{0\}\times \mathbb C$.\\
Hence,~\eqref{systeme0} can be re-written as
 \[
f(0)-\frac{1}{n(n-1)}\gamma(f)- (n-1)\mu_{n-2}(f)+\frac{1}{n}\gamma(f)\mu_{n-2}((1-id)^n)=0
 \]
 or rather
 \[
f(0)- (n-1)\mu_{n-2}(f)-\frac{1}{n(n-1)}\gamma(f)+\frac{1}{n(2n-1)}\gamma(f)=0.
 \] 
Finally, we find
 \begin{equation}\label{Y0final}
f(0)- (n-1)\mu_{n-2}(f)-\frac{1}{(n-1)(2n-1)}\gamma(f)=0.
  \end{equation}
If $Y=\{0\}\times \mathbb C$, \eqref{systeme1} can be re-written as
\begin{equation*}
\left\{
\begin{array}{rcl}
0&=&- k(0)+  (n-1)\mu_{n-2}(k)\\
n(n-1)\gamma(f)&=&(1-n) k(0)-k(1)+n k(0),
\end{array}
\right.
\end{equation*}
i.e., \eqref{Y0final} is satisfied and moreover
\[
\frac{1}{n(n-1)}\gamma(f)\stackrel{!}{=} k(0)-k(1)=-f(0)+\frac{1}{n(n-1)}\gamma (f)+f(1).
\]
Summing up, we see that necessarily
\[
\gamma(f)=(n-1)(2n-1)f(0)-(n-1)^2(2n-1)\mu_{n-2}(f)\qquad \hbox{if }Y=\{0\}^2,\\
\]
or else
\[
\begin{array}{ll}
\left\{
\begin{array}{rcl}
\gamma(f)&=&(n-1)(2n-1)f(0)-(n-1)^2(2n-1)\mu_{n-2}(f) \\
f(0)&=&f(1)
\end{array}
\right.
\qquad &\hbox{if }Y=\{0\}\times \mathbb C.
\end{array}
\]
The converse inclusion can be proven likewise, exploiting our integration-by parts-type formula as in the first part of the proof.
\end{proof}

\begin{rem}
Let us emphasize that taking into account~\cite[Thm.\ 3.1 and 3.3]{MugNic11} and Theorem~\ref{identK} one sees that the domains of $A_Y$ coincide for all $n\in \mathbb N$, in the case of $Y=\{0\}\times \mathbb C$. Comparing the special cases of $n=1$ and $n=2$ one sees that $A_Y u={\rm Id}_m^{-1}(-u'')$ for $n=1$ whereas we have just proved that $A_Y u={\rm Id}_m^{-1}(-u''+3u(0))$ if $n=2$. In general, ${\rm Id}_m^{-1} \left(\gamma(u)(1-{\rm id})^{n-2}\right)$ can be regarded as some sort of potential. In the linear case, this potential can be easily dealt with and, if desired, switched off. We will see in Section~\ref{nonlinearcase} that things are different in the nonlinear case.
\end{rem}

\begin{cor}\label{cor:mu0}
Let $n\in \mathbb N$ and define the bounded linear functional $\gamma:H^1(0,1)\to \mathbb C$ as in~\eqref{serge30/08b}. Let $\eta\in \mathbb C$. Then the operator  given by 
\[
u\mapsto {\rm Id}_m^{-1}(u'')+\eta {\rm Id}_m^{-1} \left(\gamma(u)(1-{\rm id})^{n-2}\right).
\]
either with domain
\begin{eqnarray*}
&&\{u\in H^1(0,1):\mu_0(u)=\mu_n(u)=0\}\qquad \hbox{or}\\
&&\{u\in H^1(0,1):\mu_0(u)=0,\; u(0)=u(1)\}
\end{eqnarray*}
generates on $\{f\in H^{-1}(T):\mu_0(f)=0\}$ an analytic, uniformly exponentially stable semigroup of angle $\frac{\pi}{2}$  that is immediately of trace class.
\end{cor}
\begin{proof}
The assertion follows directly from the observation that
\[
u\mapsto {\rm Id}_m^{-1} \left(\gamma(u)(1-{\rm id})^{n-2}\right)
\]
is a relatively compact perturbation of $A_Y$, and from well known perturbation results, cf.~\cite[Thm.~3.7.25]{AreBatHie01}.
\end{proof}
The following generalizes the well-posedness result~\cite[Thm.~3.7]{MugNic11}.
\begin{theo}
\label{wellp1}
Let $n\in \mathbb N$ and $\eta\in \mathbb C$. The heat-type equation 
$$\frac{\partial u}{\partial t}(t,x)=\frac{\partial^2 u}{\partial x^2}(t,x)+\eta \gamma(u) (1-x)^{n-2},\qquad t> 0,\; x\in (0,1),$$
(where $\gamma(u)$ is the same term defined in~\eqref{serge30/08b}) with moment conditions
$$\mu_0(u(t))=\mu_n(u(t))=0,\qquad t> 0,$$
or
$$\mu_0(u(t))=0,\; u(t,0)=u(t,1),\qquad t> 0,$$
and initial condition
$$u(0,\cdot)=u_0\in \{f\in H^{-1}(T):\mu_0(f)=0\}$$
is well-posed.
\end{theo}
We remark  explicitly that letting $\eta=0$ we recover well-posedness of the standard heat equation with the above conditions on the moments.

In the proof of this theorem we will need the following two results.

\begin{lemma}\label{noidm}
Let $p\in (1,\infty)$. If $f\in W^{2,p}(0,1)$, then 
\[
P_n({\rm Id}_m^{-1}(f''+\gamma(f)  (1-x)^{n-2}))=P_n (f''+\gamma (f) (1-x)^{n-2})\qquad \hbox{ in }L^2(0,1).
\]
\end{lemma}

\begin{proof}
The claim follows from~\cite[Rem.~2.6]{MugNic11} and Remark~\ref{munmun1}.(3).
\end{proof}

\begin{lemma}\label{DA2}
Let $n\ge 2$. Then the following assertions hold.
\begin{enumerate}[(1)]
\item If $Y=\{0\}^2$, then one has
\begin{equation}
D(A_Y^2)=\{u\in H^3(0,1): \mu_0(u)=\mu_0(u''-\gamma(u)  (1-{\rm id})^{n-2})=\mu_n(u)=\mu_n(u''-\gamma(u)  (1-{\rm id})^{n-2})=0\}
\label{serge19/09a}.
\end{equation}
\item If $Y=\{0\}\times \mathbb C$, then one has
\begin{equation}
D(A_Y^2)=\{u\in H^3(0,1): \mu_0(u)=\mu_0(u''-\gamma(u)  (1-{\rm id})^{n-2})=0,\; u(0)=u(1),\; u''(0)-\gamma(u)=u''(1)\}
\label{serge19/09ab}.
\end{equation}
\end{enumerate}
In either case,
$$
A_Y u=-u''+\gamma(u)  (1-{\rm id})^{n-2}, \quad\forall u\in D(A_Y^2).
$$
\end{lemma}
\begin{proof}
In either cases, the inclusion ``$\supset$'' holds because for $u$ in the right-hand side of \eqref{serge19/09a},
$-u''+\gamma(u)  (1-{\rm id})^{n-2}$ clearly belongs to $D(A_Y)$ and 
$A_Y u=-u''+\gamma(u)  (1-{\rm id})^{n-2}$, which also belongs to $D(A_Y)$.

We only prove that ``$\subset$''  holds in (1), the corresponding proof in (2) being analogous. Let us take $u\in D(A_Y^2)$. Then $u\in H^1(0,1)$ and
\[
A_Y u={\rm Id}^{-1}_m \left(-u''+\gamma(u) (1-{\rm id})^{n-2}\right)\in H^1(0,1).
\]
We first prove that $u''\in H^1(0,1)$, which clearly implies that $u\in H^3(0,1)$.
Now, set for shortness $f:=-u''+\gamma(u) (1-{\rm id})^{n-2}$ that clearly belongs to $H^{-1}(0,1)$
and consequently 
\[
\langle f,v,\rangle_{H^{-1}(0,1)-H^1_0(0,1)} 
= \langle {\rm Id}_m^{-1} f,v\rangle_{H-H^1(T)}
= \langle A_Y u,v\rangle_{H-H^1(T)}\qquad \forall v \in {\mathcal D}(0,1).
\]

Now, because by assumption $A_Y u\in H^1(0,1)$ we deduce that in fact 
\[\langle f,v,\rangle_{H^{-1}(0,1)-H^1_0(0,1)}= ( A_Y u|v)_{L^2},\]
hence $u''=-A_Y u+\gamma(u) (1-{\rm id})^{n-2}\in H^1(0,1)$ and we conclude that $u\in H^3(0,1)$, as we wanted to prove.

But now as $f$ belongs to $H^{1}(0,1)$,  by~\cite[Lemma 2.4]{MugNic11}
\[
{\rm Id}^{-1}_m f=f-\mu_0(f) \delta_1.
\]
Because both $f$ and ${\rm Id}_m^{-1} f$ and hence also $\mu_0(f)\delta_1$ belong to $H^1(0,1)$,  we have that necessarily $\mu_0(f)=0$
and therefore by Theorem~\ref{identK} and Lemma~\ref{noidm},
\[
A_Y u=-u''+\gamma(u) (1-{\rm id})^{n-2}.
\]
Accordingly, because $u\in D(A_Y^2)$ and hence $\mu_n(A_Y u)=0$, it follows that $\mu_n(u''-\gamma(u) (1-{\rm id})^{n-2})=0$. This completes the proof.
\end{proof}

\begin{proof}[Proof of Theorem~\ref{wellp1}]
We have seen that $-A_Y$ generates an analytic semigroup, hence well-posedness of the corresponding parabolic problem follows. By standard  analytic semigroup theory each initial data in $H_Y$ is immediately mapped by the semigroup into $D(A_Y^2)$. Hence, by Lemma~\ref{noidm} the claim will follow if we show that $D(A_Y^2)\subset H^2(0,1)$. But this is just one of the claims of Lemma~\ref{DA2}.
\end{proof}

\begin{rem}
Throughout this section we could have considered some perturbations of the quadratic form $a$, as we have done in~\cite[\S~3 and \S~4]{MugNic11}. In particular, for any $2\times 2$-matrix $K$ the additional term
\[
b(u,v):=\left( K \begin{pmatrix}
\mu_0(u)\\ \mu_n(u)
\end{pmatrix}\Big|  \begin{pmatrix}
\mu_0(v)\\ \mu_n(v)
\end{pmatrix} \right),\quad \forall u,v\in H^1(0,1),
\]
may be studied. For the sake of brevity, we avoid to discuss this issue thoroughly: It suffices to observe that for all $n\in \mathbb N$ there exists $C_n>0$ such that
\begin{equation}
\label{munweakcontinuous}
|\mu_n(g)|^2\leq C_n \|g\|_{L^2} \|g\|_{H^{-1}(T)},\qquad \forall g\in 
L^2(0,1):
\end{equation}
this can be shown following the proof of~\cite[Lemma~2.12]{MugNic11} and taking into account Remark~\ref{munmun1}.(1).
Accordingly, the sesquilinear form $a+b$ fits the framework of~\cite{Cro04}, hence (minus) the operator associated with this form in $H_Y$ generates a cosine operator function and an analytic semigroup of angle $\frac{\pi}{2}$.
\end{rem}

\subsection{The general case} We now complete our discussion of the linear heat equation by considering the remaining cases.

\begin{theo}\label{identK2}
Let $n\in \mathbb N$ and $Y$ be a subspace of $\mathbb C^2$, $Y\neq \{0\}^2$ and  $Y\neq \{0\}\times \mathbb C$. Then the operator $(A_Y,D(A_Y))$ associated in $H_Y$ with the quadratic form $a$  is given by
\begin{eqnarray*}
D(A_Y)&=&\left\{u\in H^1(0,1):\begin{pmatrix}
\mu_0(u)\\ \mu_n(u)
\end{pmatrix}\in Y\right\}\\
A_{Y}u&=&{\rm Id}_m^{-1}(-u''+\gamma(u) (1-{\rm id})^{n-2})-c(u)\delta_1,
\end{eqnarray*}
where $\gamma(u)$ is defined as in~\eqref{serge30/08b} and $c(u)\in \mathbb C$ is uniquely determined by the condition
\begin{equation}
\label{cuniqu}
\begin{pmatrix}
c(u)+u(1) \\
u(0)-u(1)
\end{pmatrix}\in Y^\perp
\end{equation}
\end{theo}
\begin{proof}
Again, we only treat the case $n\ge 2$ and refer the reader to~\cite[Thm.~4.3]{MugNic11} for the case $n=1$.
We denote
$${\mathcal K}_1:=\left\{u\in H^1(0,1): \begin{pmatrix}
\mu_0(u)\\ \mu_n(u)
\end{pmatrix}\in Y
 \right\}$$
 and proceed in a way similar to that in the proof of Theorem~\ref{identK} in order to determine $A_{Y}$, which by definition is given by
\begin{eqnarray*}
D(A_{Y})&:=&\{f\in V^{(n)}_{Y,2}: \exists g\in H^{-1}(T): a(f,h)=(g|h)_{H^{-1}(T)}\; \forall h\in V^{(n)}_{Y,2}\},\\
A_{Y}f&:=&g.
\end{eqnarray*}
Let us first check the inclusion $D(A_{Y})\subset {\mathcal K}_1$.
Let $f\in D(A_{Y})$. Then $f\in V^{(n)}_{Y,2}$ and there exists $g\in H^{-1}(T)$ such that
\begin{equation}\label{sn1}
(f|h)_{L^2}=\int_0^1 (P_n g)(x)(P_n\bar h)(x)\,dx
+\mu_0(g) \mu_0(\bar {h})
\quad\forall h\in V^{(n)}_{Y,2}.
\end{equation}
Now, because $g\in H^{-1}(T)$, by Lemma~\ref{lemma1}, we can consider $P_ng$  that belongs to $L^2(0,1)$.
{ Therefore by integration by parts and taking into account Lemma~\ref{lemma1}.(1) we obtain that
\begin{eqnarray*}
\int_0^1 (P_n g)(x)(P_n\bar h)(x)\,dx&=&\int_0^1 (P_n(P_n g))'(x)(P_n\bar h)(x)\,dx\\&=&
-\int_0^1 (P_n(P_n g))(x)\overline{h}(x)\,dx+[P_n(P_ng) (P_n\bar h)]_0^1\quad\forall h\in V^{(n)}_{Y,2}.
\end{eqnarray*}
}
Now, observe that the scalar number
$$\mu_0( g) \mu_0(\overline{h})+[P_n(P_ng) (P_n\bar h)]_0^1\in \mathbb C$$ 
is a linear combination of $\mu_0(\overline{h})$ and $\mu_n(\overline{h})$, hence it can be written in the form
$$c_0 \mu_0(\overline{h})+c_n \mu_n(\overline{h})=\int_0^1 (c_0+c_1(1-x)^n) \overline{h(x)}dx,$$
for some $c_0,c_n\in \mathbb C$.
Letting $\rho(x):=c_0+c_1(1-x)^n$, we obtain
 that
$$
(f|h)_{L^2}=(-P_n(P_ng)+\rho|h)_{L^2}\quad\forall h\in V^{(n)}_{Y,2}.
$$
Therefore, denoting by $\Pi$ as in the proof of Theorem~\ref{identK} the orthogonal projection of $L^2(0,1)$ onto the vector space spanned by $1$ and $(1-{\rm id})^n$, we obtain (by restricting the previous identity to all $h\in {V^{(n)}_{\{0\}^2,2}\subset V^{(n)}_{Y,2}}$)
$$
(I-\Pi)(f+P_n(P_ng)-\rho)=0,
$$
or equivalently
\begin{equation}\label{sn2}
f=(I-\Pi)(-P_n(P_ng)+\rho)+\Pi f=-P_n(P_ng)+\Pi (P_n(P_ng)+f).
\end{equation}

This proves that $f$ belongs to $H^1(0,1)$ and (differentiating~\eqref{sn2} twice)
that 
\begin{equation}\label{sn2bis}
g=-f''+\gamma (1-{\rm id})^{n-2},
\end{equation}
 in the distributional sense (i.e. in ${\mathcal D}'(0,1)$) for some $\gamma\in \mathbb C$. Hence, by
\cite[Lemma 2.4]{MugNic11}, there exists $c(f)\in \mathbb C$ such that
\[
A_{Y}f=g={\rm Id}_m^{-1}(-f''+\gamma (1-{\rm id})^{n-2})-c\delta_1,
\]
and in fact $c(f)=-\mu_0(g)$.

It remains to check the condition \eqref{cuniqu}. 
But we first notice that, for all $f\in D(A_{Y})$,~\eqref{sn2} leads to
$$
f'=-P_ng-\frac{\gamma}{n-1} (1-{\rm id})^{n-1},
$$
for the same $\gamma$ as in \eqref{sn2bis}. By~\eqref{sn1} we obtain
\begin{eqnarray*}
\int_0^1f(x) \bar h(x)\,dx&=&\int_0^1 (-f'(x)-\frac{\gamma}{n-1} (1-x)^{n-1})(P_n\bar h)(x)\,dx-c(f) \mu_0(\overline{h})
\\
&=&-\int_0^1 f'(x)(P_n\bar h)(x)\,dx-\frac{\gamma}{n-1} \mu_{n-1}(P_n\bar h)-c(f) \mu_0(\overline{h})\quad\forall h\in V^{(n)}_{Y,2}.
\end{eqnarray*}
As $\mu_{n-1}(P_n\bar h)=0$ by Remark~\ref{munmun1}.(2), we deduce that
\begin{equation}\label{serge19/09b}
\int_0^1f(x) \overline{h}(x)\,dx=-\int_0^1 f'(x)(P_n\bar h)(x)\,dx-c(f) \mu_0(\overline{h})\quad\forall h\in V^{(n)}_{Y,2}.
\end{equation}
Integrating by parts the first term on the right-hand side we obtain
\begin{eqnarray*}
-\int_0^1 f'(x)(P_n\bar h)(x)\,dx&=&\int_0^1 f(x) \bar h(x)\,dx+f(0) (P_n\bar h)(0)-f(1) (P_n\bar h)(1)\\
&=&\int_0^1 f(x) \bar h(x)\,dx-f(0) \mu_n(\bar h)-f(1) (\mu_0(\overline{h})-\mu_n(\bar h)),
\end{eqnarray*}
owing to Lemma~\ref{lemma1}.(1) and~\eqref{Pn02}, respectively.
Plugging this identity in \eqref{serge19/09b} we find that
$$
-(f(1)+c(f)) \mu_0(\overline{h})
+(f(1)-f(0)) \mu_n(\bar h)=0\quad\forall h\in V^{(n)}_{Y,2}.
$$
 By the surjectivity result from Lemma~\ref{lemma:surj}, we have shown \eqref{cuniqu}: Let us notice that \eqref{cuniqu} determines in a unique way $c$.  

We now prove the converse inclusion.
Let then $f\in {\mathcal K}_1$. Then we can take 
\[
g:={\rm Id}_m^{-1}(-f''+\gamma (1-{\rm id})^{n-2})-c\delta_1,
\]
with $c\in \mathbb C$ fixed by the condition \eqref{cuniqu} (and is equal to $-\mu_0(g)$) and $\gamma$ will be fixed later on.
Hence by definition of the inner product in $H^{-1}(T)$ and the fact that $P_n\delta_1=0$ by Remark~\ref{munmun1}.(3), we will have for any $h\in V^{(n)}_{Y,2}$
\begin{eqnarray*}
(g|h)_{H^{-1}(T)}&=&(P_n({\rm Id}_m^{-1}(-f''+\gamma (1-{\rm id})^{n-2}))|P_n  h)_{L^2}-c\mu_0(\overline{h})\\
&=&-(P_n({\rm Id}_m^{-1}f''| P_n  h)_{L^2}    +\gamma (P_n(1-{\rm id})^{n-2})| P_n  h)_{L^2}-c\mu_0(\overline{h}).
\end{eqnarray*}
But simple calculations and an integration by parts yield
\begin{equation}\label{serge14/11}
(P_n(1-{\rm id})^{n-2})| P_n  h)_{L^2}=\frac{n}{(n-1)(2n-1)}(\mu_1(\bar h)-\mu_n(\bar h)).
\end{equation}
Hence by 
Theorem~\ref{ibp}, for all $h\in V^{(n)}_{Y,2}$  we get
\begin{equation}\label{serge20/09}
(g|h)_{H^{-1}(T)}=(f|h)_{L^2}+\begin{pmatrix}
-(f(1)+c)\\ -n f(0)+ n (n-1)\mu_{n-2}(f)+\frac{n\gamma }{(n-1)(2n-1)}\\
-f(0)+f(1)-n (n-1)\mu_{n-2}(f)+n f(0)-\frac{n\gamma }{(n-1)(2n-1)}\\  
\end{pmatrix}\begin{pmatrix}
\mu_0(h)\\
 \mu_1(h)\\ \mu_n (h)
\end{pmatrix}.
\end{equation}
From this identity we see that $f$ will belong to $D(A_Y)$ if the second entry of the second term in the right-hand side of this identity is zero.
This motivates us to choose $\gamma$ such that 
\[
-n f(0)+ n (n-1)\mu_{n-2}(f)+\frac{n\gamma }{(n-1)(2n-1)}=0,
\]
which is equivalent to ~\eqref{serge30/08b}. With this choice we see that \eqref{serge20/09} is equivalent to
\[
(g|h)_{H^{-1}(T)}=(f|h)_{L^2}+\begin{pmatrix}
-(f(1)+c)\\
-f(0)+f(1)\\  
\end{pmatrix}\begin{pmatrix}
\mu_0(h)\\
 \mu_n (h)
\end{pmatrix}=(f|h)_{L^2}, \forall h\in V^{(n)}_{Y,2},
\]
the second identity following from \eqref{cuniqu}.  
This shows that
$$
a(f,h)=(g|h)_{H^{-1}(T)} ,\qquad \forall h\in V^{(n)}_{Y,2},
$$
and proves that $f$ belongs to $D(A_{Y})$.
\end{proof}

\begin{cor}
Under the assumptions of Theorem \ref{identK2}, let $\eta\in \mathbb C$. Then the operator  $A_Y$ given by 
\begin{eqnarray*}
D(A_Y)&:=&\left\{u\in H^1(0,1):\begin{pmatrix}
\mu_0(u)\\ \mu_n(u)
\end{pmatrix}\in Y\right\},\\
A_Y u&:=& {\rm Id}_m^{-1}(u'')+\eta\gamma(u)  {\rm Id}_m^{-1} \left((1-{\rm id})^{n-2}\right),
\end{eqnarray*}
generates on $H_Y= H^{-1}(T)$ an analytic, uniformly exponentially stable semigroup of angle $\frac{\pi}{2}$  that is immediately of trace class.
\end{cor}
\begin{proof}
Again the assertion follows directly from the observation that
\[
u\mapsto \gamma(u) {\rm Id}_m^{-1} \left((1-{\rm id})^{n-2}\right)
\]
is a relatively compact perturbation of $A_Y$.
Hence a well known perturbation result, cf.~\cite[Thm.~3.7.25]{AreBatHie01}, allows us to conclude.
\end{proof}

As a consequence we obtain the following existence result.
\begin{theo}
\label{wellp1general}
Let $n\in \mathbb N$ and
 $\eta, y\in \mathbb C$. Then the heat equation
$$\frac{\partial u}{\partial t}(t,x)=\frac{\partial^2 u}{\partial x^2}(t,x)+\eta \gamma(u(t,\cdot)) (1-x)^{n-2},\qquad t> 0,\; x\in (0,1),$$
(where $\gamma(u)$ is the same term defined in~\eqref{serge30/08b}) with moment conditions
$$\mu_n(u(t,\cdot))=y\mu_0(u(t,\cdot)),\qquad t> 0,$$
\[
-\mu_n(u''(t,\cdot))+\frac{\gamma(u(t,\cdot))}{2n-1}=y(-\mu_0(u''(t,\cdot))+\frac{\gamma(u(t,\cdot))}{n-1}),\qquad t> 0,
\]
and
$$
-\mu_0(u''(t,\cdot))+\frac{\gamma(u(t,\cdot))}{n-1}=(u(t,1)-u(t,0))\bar y-u(t,1)
\qquad t> 0,$$
and initial condition
$$u(0,\cdot)=u_0\in  H^{-1}(T)$$
is well-posed.
\end{theo}
\begin{proof}
It suffices to apply the previous Corollary with $Y$ spanned by $(1,y)^\top$ and Lemma ~\ref{DA2gen} below.
Note that the condition \eqref{cuniqu} is   then equivalent to
\[
c(u)=-u(1)-
(u(0)-u(1)) \bar y.
\]
This concludes the proof.
\end{proof}

\begin{lemma}\label{DA2gen}
Under the assumptions of Theorem \ref{identK2} one has
\begin{eqnarray*}
\label{serge14/12}
D(A_Y^2)&=&\bigg\{u\in H^3(0,1): \mu_n(u)=y \mu_0(u),\\
&&\quad \mu_n(u''-\gamma(u)  (1-{\rm id})^{n-2})=y
\mu_0(u''-\gamma(u)  (1-{\rm id})^{n-2}),\\
&&\quad \mu_0(u''-\gamma(u)  (1-{\rm id})^{n-2})=
(u(1)-u(0))\bar y-u(1)\bigg\},
\end{eqnarray*}
when $Y$ is spanned by $(1,y)^\top$ 
and 
$$
A_Y u=-u''+\gamma(u)  (1-{\rm id})^{n-2}, \quad\forall u\in D(A_Y^2).
$$
\end{lemma}
\begin{proof}
The inclusion ``$\supset$'' holds because for $u$ in the right-hand side of \eqref{serge14/12},
$-u''+\gamma(u)  (1-{\rm id})^{n-2}$ clearly belongs to $D(A_Y)$ and 
$A_Y u=-u''+\gamma(u)  (1-{\rm id})^{n-2}$, which also belongs to $D(A_Y)$.

We now prove the converse inclusion ``$\subset$''. Let us fix $u\in D(A_Y^2)$. Then $u\in H^1(0,1)$ and
\[
A_Y u={\rm Id}^{-1}_m \left(-u''+\gamma(u) (1-{\rm id})^{n-2}\right)-c(u)\delta_1\in H^1(0,1).
\]
This identity implies that

\[
-u''=A_Yu-\gamma(u) (1-{\rm id})^{n-2} \hbox{ in } {\mathcal D}'(0,1),
\]
and the assumption $A_Y u\in H^1(0,1)$ implies that $u\in H^3(0,1)$.
With this regularity at hands, by ~\cite[Rem.~2.6]{MugNic11} we see that
\[
A_Y u=-u''+\gamma(u) (1-{\rm id})^{n-2})+\left(\mu_0(u'')-\gamma(u) \mu_0((1-{\rm id})^{n-2})
-c(u)\right)\delta_1\in H^1(0,1).
\]

Hence necessarily 
\[
\mu_0(u'')-\gamma(u) \mu_0((1-{\rm id})^{n-2})
-c(u)=0,
\]
and  
\[
A_Y u=-u''+\gamma(u) (1-{\rm id})^{n-2}.
\]

Because $u\in D(A_Y^2)$, $-u''+\gamma(u) (1-{\rm id})^{n-2}$ belongs to $D(A_Y)$ and  the characterization of $D(A_Y)$ completes the proof.
\end{proof}

\section{The porous medium equation}\label{nonlinearcase}

Throughout this section, we consider real valued functions spaces and we let 
\begin{itemize}
\item $p\in (1,\infty)$,
\item $Y$ be a subspace of $\mathbb R^2$,
\item $n\in \mathbb N$.
\end{itemize}
We consider the functional ${\mathcal E}_Y:H_Y\to \mathbb R_+\cup \{+\infty\}$ defined by
$${\mathcal E}_Y:f\mapsto 
\left\{
\begin{array}{ll}
\frac{1}{p}\int_0^1 |f(x)|^p dx,\qquad &\hbox{if } f\in V^{(n)}_{Y,p},\\
+\infty, &\hbox{otherwise,}
\end{array}
\right.$$
where $H_Y$ and $V^{(n)}_{Y,p}$ are the spaces introduced in~\eqref{defi_hy} and~\eqref{defi:V^n}. It turns out that the subdifferential of ${\mathcal E}_Y$ has particularly good properties.


\begin{theo}\label{prop_functi}
Let $Y$ be a subspace of $\mathbb R^2$.  
 Then the functional {$\mathcal E_Y$ is proper and convex.
Furthermore, $\mathcal E_Y$ is continuously Fr\'echet differentiable as a functional on $V^{(n)}_{Y,p}$ and lower semicontinuous as a functional on $H_Y$
.}
\end{theo}


\begin{proof}
It is clear that $\mathcal E_Y$ is proper and convex. Moreover, it is continuously differentiable -- in fact, the Fr\'echet derivative of $\mathcal E_Y$ is given by
\begin{equation}
\label{eprime}
\mathcal E_Y'(f)h=\int_0^1 |f(x)|^{p-2}f(x)h(x) \, dx, \quad f,h \in V^{(n)}_{Y,p}.
\end{equation}

Thus, $\mathcal E_Y$ is in particular lower semicontinuous as a functional on $V^{(n)}_{Y,p}$, and lower semicontinuity as a functional on $H_Y$ follows from~\cite[Lemma~IV.5.2]{Sho97}.
\end{proof}

\begin{rem}
An additional term of the form $\Phi(\begin{pmatrix} \mu_0(f)\\ \mu_n(f)\end{pmatrix})$, where $\Phi:Y\to \mathbb R$ is a convex continuously differentiable function, may be easily dealt with by means of simple perturbation results for nonlinear forms, see  e.g.~\cite{MugPro11}.
\end{rem}

\begin{theo}\label{identkp}
Let $\gamma(\cdot)$ be defined by \eqref{serge30/08b}. Then the following assertions hold.
\begin{enumerate}[(1)]
\item 
If either $Y=\{0\}^2$ or $Y=\{0\}\times \mathbb R$, then the subdifferential $\partial{\mathcal E}_{Y}$ of $\mathcal E_{Y}$ with respect to $H_Y$ agrees with the nonlinear operator $A_{Y}$ given by\begin{eqnarray*}
D(A_Y)&:=&\{f\in L^p(0,1): |f|^{p-2}f\in H^1(0,1):\mu_0(f)=\mu_n(f)=0\}\qquad \hbox{or}\\
D(A_Y)&:=&\{f\in L^p(0,1): |f|^{p-2}f\in H^1(0,1):\mu_0(f)=0,\; f(0)=f(1)\}
\end{eqnarray*}
respectively, whose action is given in both cases by
\[
A_Y f:={\rm Id}_m^{-1}(-(|f|^{p-2}f)''+\gamma(f) (1-{\rm id})^{n-2}).
\]
\item If $Y\not=\{0\}^2$ and $Y\not=\{0\}\times \mathbb R$, then
the subdifferential $\partial{\mathcal E}_{Y}$ of $\mathcal E_{Y}$ with respect to $H_Y$ agrees with the nonlinear operator $A_{Y}$ given by
\begin{eqnarray*}
D(A_Y)&:=& \left\{f\in L^p(0,1): |f|^{p-2}f\in H^1(0,1)\hbox{ and } \;\begin{pmatrix}\mu_0(f)\\ \mu_n(f)\end{pmatrix}\in Y\right\},\\
\end{eqnarray*}
whose action is given by
\[
A_Y f:={\rm Id}_m^{-1}(-(|f|^{p-2}f)''+\gamma(|f|^{p-2}f) (1-id)^{n-2}) -c(f)\delta_1.
\]
Here  $c(f)\in \mathbb R$ is uniquely determined by the condition
\begin{equation}
\label{cuniqu-pme}
\begin{pmatrix}
c(f)+|f|^{p-2}(1)f(1) \\
|f|^{p-2}(0)f(0)-|f|^{p-2}(1)f(1)
\end{pmatrix}\in Y^\perp.
\end{equation}
\end{enumerate}
\end{theo}
Here, by definition, for a given reflexive Banach space $X$ the subdifferential of a functional $\mathcal E$ with respect to a Hilbert space $H$ such that $X$ is continuously and densely embedded in $H$ is the single-valued operator given by
\begin{eqnarray*}
D(\partial \mathcal{E})&:=&\{f \in H_Y: \exists \, g \in H_Y \text{ s.t. } \mathcal{E}^\prime(f)h =(g|h)_{H_Y} \;\forall h \in V^{(n)}_{Y,p}\}, \\
 \partial \mathcal{E}(f)&:=&g,
\end{eqnarray*} 
cf.~\cite[Lemma~2.8.9]{Nit10}.

\begin{proof}
We do not deliver the proof, which can be performed closely following those of Theorems~\ref{identK} and~\ref{identK2}  if $n\ge 2$, or rather the proofs of~\cite[Thm.~3.1, 3.3, 4.3]{MugNic11} if $n=1$, up to replacing $f$ by $|f|^{p-2}f$ throughout. The only noteworthy modifications are the following ones: On one hand, one has to replace the orthogonal projection  $\Pi \psi\in L^2 $ of a vector $\psi\in L^2(0,1)$ onto the space of polynomials spanned by  $1$ and $(1-{\rm id})^n$ by the vector  $\Pi \psi \in L^{p'}(0,1)$ of best approximation of $\psi\in L^{p'}(0,1)$ in the subspace of ${\mathbb R}[x]$ spanned by  $1$ and $(1-{\rm id})^n$:  This makes sense even if $\Pi$ is not an \emph{orthogonal} projection, and of course ${\mathbb R}[x]=\Pi {\mathbb R}[x] + (I-\Pi){\mathbb R}[x]$ is dense in $L^p(0,1)$. On the other hand, one must check directly that each $f\in D(\partial {\mathcal E})$ is of class $L^p(0,1)$ -- this can in turn be done observing that because by assumption $|f|^{p-2}f\in H^1(0,1)\hookrightarrow L^\frac{p}{p-1}(0,1)$, $f^p\in L^1(0,1)$, i.e., $f\in L^p(0,1)$.
\end{proof}


By virtue of Theorems\ref{prop_functi} and~\ref{identkp}, the following assertion is a direct consequence of the general theory of nonlinear semigroups associated with subdifferentials, see e.g.\ ~\cite[Prop.~IV.5.2]{Sho97} and~\cite[Th\'eo.~3.2]{Bre73}.

\begin{prop}\label{ralphwell}
Let $p\in (1,\infty)$. Let $T>0$, $u_0\in V^{(n)}_{Y,p}$ and $f\in L^2(0,T;H_Y)$. Then the abstract Cauchy problem 
\begin{equation}\label{acpy}
\left\{
\begin{array}{rcll}
\frac{d u}{d t}(t)&=&-A_Y u(t)+f(t),\qquad & t>0,\\
u(0)&=&u_0,
\end{array}
\right.
\end{equation}
admits a unique solution $u\in H^1(0,T;H_Y)\cap L^\infty(0,T;V^{(n)}_{Y,p})$.

If $f\equiv0$, then the the following additional assertions hold:
\begin{itemize}
\item the solution of~\eqref{acpy} is given by a strongly continuous semigroup of nonlinear contractions,
\item $u(t,\cdot)\in D(A_Y)$ for all $t>0$, 
\item $u\in C([0,\infty);H_Y)$ is everywhere right differentiable and satisfies~\eqref{acpy} -- and in particular the initial condition -- for all $t\ge 0$, and finally
\item $t\mapsto \|u(t)\|_{L^p}^p$ is convex and monotonically decreasing on each interval $[\delta,\infty)$ for all $\delta>0$.
\end{itemize}
\end{prop}

\begin{rem}
Observe that $A_Y$ is \emph{not} the operator associated with the (non-signed) PME. Hence, unlike in the setting discussed in Section~\ref{linearcase}, in the present nonlinear case we are not able to show that the Cauchy problem associated with $-A_Y$ takes the form of a common PME (possibly with potential), because the general theory of subdifferentials apparently does not ensure enough regularity of solutions to allow us to drop the term ${\rm Id}_m^{-1}$. However, taking into account Remark~\ref{munmun1}.(5) one observes that $A_Y$ does indeed act on twice weakly differentiable functions as a second order differential operator with a potential term.  

{Observe that a possible, weaker but still sufficient approach would consist in restricting ourselves to consider initial values in $V^{(n)}_{Y,p}$, or some other subspace of $L^2(0,1)$ that is left invariant under the semigroup; and then showing that the restriction of the semigroup is strongly continuous, and that its generator is a restriction of $A_Y$ to some subspace of $H^2(0,1)$. In view of Lemma~\ref{noidm}, this would do the job. Unfortunately, it seems to be unclear to which extent this kind of procedure is allowed by the general theory of nonlinear semigroups. However, it is for instance known that the operators associated with the PME on bounded domains $\Omega$ ``generate semigroups of order-preserving contractions in $L^1(\Omega)$'', cf. \cite[\S~20.1.2]{Vaz07}.}
\end{rem}

}

{We have already seen that in the linear case the problem is governed by a uniformly exponentially stable semigroup on $H_Y$. This means that the $H_Y$-norm of each solution tends to 0 as $t\to \infty$, uniformly for all initial data. What about the nonlinear case? We alreay know from Theorem~\ref{ralphwell} that if the inhomogeneous term $f\equiv0$, then the $L^p$-norm of the solution is decreasing. Further information can be obtained. Inspired by the classical analysis of the PME, see e.g.~\cite[Thm.~11.9]{Vaz07}, we obtain the following, where for the sake of simplicity we restrict to the case of $Y=\{0\}^2$.
\begin{prop}
Consider the abstract Cauchy problem~\eqref{acpy} with $u_0\in V^{(n)}_{Y,p}$ and $f=0$. If $Y=\{0\}^2$, then the solution $u$ given by Theorem~\ref{ralphwell} satisfies
\[
\|u(t)\|_{H_Y}^2\le Kt^{-\frac{2}{p-2}}\qquad \hbox{if }p>2,
\]
for some $K>0$ only depending on $p$, whereas
\[
\|u(t)\|_{H_Y}^2\le \|u(0)\|_{H_Y}^2 e^{-Kt}\qquad \hbox{if }p\le 2,
\]
for some $K>0$ proportional to $\|u(0)\|_{H_Y}^{\frac{p-2}{2}}$.
\end{prop}
Observe that for $p=2$ we recover information already yield from the exponential stability of the linear semigroup (cf.~Corollary~\ref{cor:mu0}), whereas for $p<2$ (the case of the FDE) we deduce exponential decay -- with decay rate depending on the initial data.
}

\begin{proof}
By the definition of the inner product in $H_Y$ and since by assumption $\mu_0(u(t))=\mu_n(u(t))=0$, we find that
\begin{eqnarray*}
\left(u(t)\big|(1-{\rm id})^{n-2}\right)_{H_Y}&=&\int_0^1\int_0^x u(t,y)dy(1-x)^{n-1}dx-\frac{n}{(n-1)(2n-1)} \mu_1(u(t))\\
&=&\frac1n \mu_n(u(t))-\frac{n}{(n-1)(2n-1)} \mu_1(u(t))\\
&=&-\frac{n}{(n-1)(2n-1)} \mu_1(u(t)),
\end{eqnarray*}
where the second identity follows integrating by parts.
Hence for all $t\ge 0$
\begin{eqnarray*}
\frac12 \frac{{d} \|u(t)\|_{H_Y}^2}{dt}
&=& \left(u(t)|\frac{du}{dt}(t) \right)_{H_Y}\\
&=& -\left(u(t)|A_Y u(t) \right)_{H_Y}\\
&=& -\left(u(t)|{\rm Id}_m^{-1}(-(|u(t)|^{p-2}u(t))''+\gamma(u(t)) (1-{\rm id})^{n-2})\right)_{H_Y}\\
&=& -\left(u(t)|(|u(t)|^{p-2}u(t))\right)_{L_2}-\gamma(u(t)) \left(u(t)\big|(1-{\rm id})^{n-2}\right)_{H_Y}\\
&=&-\|u(t)\|^p_{{V^{(n)}_{Y,p}}}+\left(u(t,0) -(n -1)\mu_{n-2} (u(t))\right)
(n-1)(2n-1) \int_0^1\int_0^x u(t,y)dy(1-x)^{n-1}dx.
\end{eqnarray*}
if $n\ge 2$, and a similar but simpler computations holds for $n=1$.
We finally arrive at
$$
\frac12 \frac{{d} \|u(t)\|_{H_Y}^2}{dt}
=-\|u(t)\|^p_{{V^{(n)}_{Y,p}}}.
$$

Because $V^{(n)}_{Y,p}\hookrightarrow H_Y$, we get
$$
\frac12 \frac{d \|u(t)\|_{H_Y}^2}{d t}
\leq-C_0\|u(t)\|^p_{H_Y},
$$
for some $C_0>0$.
Setting 
\[
v(t):=\|u(t)\|_{H_Y}^2,\qquad t\ge 0,
\]
we have shown that
\begin{equation}\label{serge19/12}
v'(t)\leq -C v(t)^\alpha,
\end{equation}
with 
\[\alpha:=\frac{p}{2}
\]
and $C:=2C_0$. Now we distinguish the case $\alpha>1$ and $\alpha\leq 1$. 

1) In the first case we consider the function $w$ defined by
$$
w(t):=C (\alpha-1) t-v(t)^{1-\alpha},\qquad t\ge 0.
$$
(Without loss of generality we can assume that
$v(t)>0$, for all $t>0$, otherwise as $v$ is decaying, $v$ would be zero after a time $t_0>0$).

By direct calculations, we have
$$
w'=(\alpha-1) (C+v^{-\alpha} v'),\qquad \forall t> 0,
$$
and by \eqref{serge19/12}, we get that $w$ is decaying.
Hence 
$$
w(t)\leq w(0),\quad\forall t>0,
$$
or equivalently
$$
v(t)\leq K t^{-\frac{1}{\alpha-1}},\quad\forall t>0,
$$
with $K:=\left(C(\alpha-1)\right)^{-\frac{1}{\alpha-1}}>0$, which shows the polynomial decay.

2) If $\alpha\leq 1$, then again by the decay of $v$, we have
$$
v(t)\leq v(0)^{1-\alpha} v(t)^\alpha,\quad\forall t>0,
$$
and therefore it follows from \eqref{serge19/12} that
\begin{equation}\label{serge19/12b}
v'(t)\leq -K v(t)^\alpha,
\end{equation}
with $K:=\frac{C}{v(0)^{1-\alpha}}$.
As before this property implies that the mapping $t\mapsto v(t) e^{Kt}$ is decaying 
and therefore we obtain
$$
v(t)\leq v(0) e^{-Kt},\quad\forall t>0,
$$
which shows the claimed exponential decay.
\end{proof}

\bibliographystyle{alpha}
\bibliography{../../referenzen/literatur}
\end{document}